\documentclass{scrartcl}

\usepackage{etex}
\usepackage[T1]{fontenc}
\usepackage[latin1]{inputenc}
\usepackage[intlimits]{amsmath}
\usepackage{amssymb}
\usepackage{amsthm}
\usepackage{dsfont}
\usepackage{pictex}

\theoremstyle{plain}\newtheorem{defi}{Definition}[section]
\theoremstyle{plain}\newtheorem{prop}[defi]{Theorem}
\theoremstyle{plain}\newtheorem{lem}[defi]{Lemma}
\theoremstyle{plain}\newtheorem{cor}[defi]{Corollary}
\theoremstyle{definition}\newtheorem{exa}[defi]{Example}
\theoremstyle{definition}\newtheorem{rem}[defi]{Remark}

\DeclareMathOperator{\var}{Var}
\DeclareMathOperator{\cov}{Cov}

\begin{document}

\title{Central Limit Theorem and the Bootstrap for U-Statistics of Strongly Mixing Data}
\author{Herold Dehling\thanks{Fakult\"{a}t f\"{u}r Mathematik, Ruhr-Universit\"{a}t Bochum, 44780 Bochum, Germany}, Martin Wendler\thanks{Corresponding Author, Fakult\"{a}t f\"{u}r Mathematik, Ruhr-Universit\"{a}t Bochum, 44780 Bochum, Germany, Fax: +49-234-3214039,  Email address:  Martin.Wendler@rub.de}}

\maketitle

\begin{abstract}
The asymptotic normality of U-statistics has so far been proved for iid data and under various mixing conditions such as absolute regularity, but not for strong mixing. We use a coupling technique introduced in 1983 by Bradley \cite{brad} to prove a new generalized covariance inequality similar to Yoshihara's \cite{yosh}. It follows from the Hoeffding-decomposition and this inequality that U-statistics of strongly mixing observations converge to a normal limit if the kernel of the U-statistic fulfills some moment and continuity conditions.

The validity of the bootstrap for U-statistics has until now only been established in the case of iid data (see Bickel and Freedman \cite{bick}). For mixing data, Politis and Romano \cite{poli} proposed the circular block bootstrap, which leads to a consistent estimation of the sample mean's distribution. We extend these results to U-statistics of weakly dependent data and prove a CLT for the circular block bootstrap version of U-statistics under absolute regularity and strong mixing. We also calculate a rate of convergence for the bootstrap variance estimator of a U-statistic and give some simulation results.

\end{abstract}

\section{U-Statistic CLT}

U-statistics are a broad class of nonlinear functionals, including many well-known examples such as the variance estimator or the Cramer-von Mises-sta\-tis\-tic. For simplicity of notation, we concentrate on the case of bivariate U-statistics.

\begin{defi}A U-statistic with a symmetric and measurable kernel $h:\mathbb{R}^2\rightarrow\mathbb{R}$ is defined as
\begin{equation*}
 U_{n}\left(h\right)=\frac{2}{n\left(n-1\right)}\sum_{1\le i<j\leq n}h\left(X_{i},X_{j}\right).
\end{equation*}
\end {defi}

$U_{n}\left(h\right)$ is the uniformly minimum variance estimator of $\theta=E\left[h\left(X_{1},X_{2.}\right)\right]$, if $X_{1},\ldots,X_{n}$ are iid with an arbitrary absolutely continuous distribution. To prove asymptotic normality of U-statistics, Hoeffding \cite{hoef} decomposed $U_{n}\left(h\right)$ as follows:
\begin{equation*}
U_{n}\left(h\right)=\theta+\frac{2}{n}\sum_{i=1}^{n}h_{1}\left(X_{i}\right)+\frac{2}{n\left(n-1\right)}\sum_{1\leq i<j\leq n}h_{2}\left(X_{i},X_{j}\right)
\end{equation*}
with
\begin{align*}
h_1(x)&:=Eh(x,X_{2})-\theta \\
h_2(x,y)&:=h(x,y) - h_1(x) -h_1(y) -\theta.
\end{align*}
The linear part $\frac{2}{\sqrt{n}}\sum_{i=1}^{n}h_{1}\left(X_{i}\right)$ is a sum of iid random variables with a normal limit distribution,  $\frac{2}{\sqrt{n}\left(n-1\right)}\sum_{1\leq i<j\leq n}h_{2}\left(X_{i},X_{j}\right)=U_n\left(h_2\right)$ is called the degenerate part of the U-statistic and converges to zero in probability, as its parts are uncorrelated, so the U-statistic is asymptotically normal.

Under dependence, the summands of the degenerate part can be correlated and this can change the limit distribution. Under the strong assumption of $\star$-mixing, Sen \cite{sen} showed that U-statistics are asymptotically normal. Yoshihara assumed $X_{1},\ldots,X_{n}$ to be stationary and absolutely regular  and proved a CLT for U-Statistics under this weaker condition (for a detailed description of the various mixing conditions see Doukhan \cite{douk} and Bradley \cite{bra2}).
\begin{defi}
A sequence $\left(X_{n}\right)_{n\in\mathbb{N}}$ of random variables is called absolutely regular, if
\begin{equation*}
\beta\left(m\right):=\sup\left\{\beta\left(\left(X_{1},\ldots,X_{k}\right),\left(X_j\right)_{j\geq k+m}\right)|k\in\mathbb{N}\right\}\xrightarrow{m\rightarrow\infty}0
\end{equation*}
where $\beta$ is the absolute regularity coefficient defined as
\begin{equation*}
\beta\left(Y,Z\right):=E\left[\sup_{A\in\sigma\left(Y\right)}\left|P\left[A|Z\right]-P\left[A\right]\right|\right].
\end{equation*}
\end{defi}

Yoshihara has proved the asymptotic normality of the U-statistic $U_{n}\left(h\right)$ using a generalized covariance inequality: With increasing distance between the indices $i_{1},i_{2},i_{3},i_{4}$, the covariance of $h_{2}\left(X_{i_{1}},X_{i_{2}}\right)$ and $h_{2}\left(X_{i_{3}},X_{i_{4}}\right)$ becomes smaller and therefore the degenerate part vanishes as in the independent case.

Denker and Keller \cite{den2} have weakened the mixing assumption to functionals of absolutely regular processes, Borovkova, Burton and Dehling \cite{boro} showed convergence of the empirical U-process to a Gaussian process. Recently, Hsing and Wu \cite{hsin} proved a CLT for weighted U-statistics of processes that have the form $X_n=F\left(\ldots,\epsilon_{n-2},\epsilon_{n-1},\epsilon_n\right)$, where $\left(\epsilon_{n}\right)_{n\in\mathbb{Z}}$ is an i.i.d. process.

We want to extend Yoshihara's CLT to random variables, which fullfill the strong mixing condition:
\begin{defi} A sequence $\left(X_{n}\right)_{n\in\mathbb{N}}$ of random variables is called strong mixing if
\begin{equation*}
 \alpha\left(m\right)=\sup\left\{\alpha\left(\left(X_{1},\ldots,X_{k}\right),\left(X_j\right)_{j\geq k+m}\right)|k\in\mathbb{N}\right\}\xrightarrow{m\rightarrow\infty}0
\end{equation*}
where $\alpha$ is the strong mixing coefficient defined as
\begin{equation*}
 \alpha\left(Y,Z\right)=\sup_{\substack{A\in\sigma\left(Y\right)\\B\in\sigma\left(Z\right)}}\left|P\left(A\cap B\right)-P\left(A\right)P\left(B\right)\right|.
\end{equation*}
\end{defi}
Strong mixing is weaker than absolute regularity, but absolute regularity and strong mixing are equivalent for random variables, which take their values in a finite set. One can approximate general random variables by such discrete ones. To make this discretization work for U-Statistics, we impose a continuity condition on the kernel, that is not needed in the case of absolutely regular data:

\begin{defi}\label{def1} Let $\left(X_{n}\right)_{n\in\mathbb{N}}$ be a stationary process. A kernel $h$ is called $\mathcal{P}$-Lipschitz-continuous if there is a constant $L>0$ with
\begin{equation*}
 E\left[\left|h\left(X,Y\right)-h\left(X',Y\right)\right|\mathds{1}_{\left\{\left|X-X'\right|\leq\epsilon\right\}}\right]\leq L\epsilon
\end{equation*}
for every $\epsilon>0$, every pair $X$ and $Y$ with the common distribution $\mathcal{P}_{X_{1},X_{k}}$ for a $k\in\mathbb{N}$ or $\mathcal{P}_{X_{1}}\times\mathcal{P}_{X_{1}}$ and $X'$ and $Y$ also with one of these common distributions.
\end{defi}

$\mathcal{P}$-Lipschitz-continuity is a special case of $p$-continuity established by Borov\-kova, Burton and Dehling \cite{boro}. It is clear that every Lipschitz-continuous kernel is $\mathcal{P}$-Lipschitz-continuous. But this definition holds also for many kernels that are not Lipschitz-continuous in the ordinary sense:

\begin{exa}[Variance estimation]\label{ex1} Consider stationary random variables with bounded variance and the kernel $h\left(x,y\right)=\frac{1}{2}\left(x-y\right)^{2}$. The related U-statistic is the well known variance estimator
\begin{equation*}
 U_{n}\left(h\right)=\frac{1}{n-1}\sum_{i=1}^{n}\left(X_{i}-\bar{X}\right)^{2}.
\end{equation*}
For random variables $X$, $X'$ and $Y$ as above, we get:
\begin{align*}
 &E\left[\left|\frac{1}{2}\left(X-Y\right)^{2}-\frac{1}{2}\left(X'-Y\right)^{2}\right|\mathds{1}_{\left\{\left|X-X'\right|\leq\epsilon\right\}}\right]\\
=&\frac{1}{2}E\left[\left|X-X'\right|\left|X+X'-2Y\right|\mathds{1}_{\left\{\left|X-X'\right|\leq\epsilon\right\}}\right]\\
&\leq\frac{1}{2}\epsilon E\left[\left|X+X'-2Y\right|\right]\leq2\epsilon E\left|X\right|
\end{align*}
This proves the $\mathcal{P}$-Lipschitz-continuity of $h$.
\end{exa}

\begin{exa}[Dimension estimation] Let $t>0$. The kernel $h\left(x,y\right)=\mathds{1}_{\left\{\left|x-y\right|<t\right\}}$ is related to the Grassberger-Procaccia dimension estimator \cite{grass}. It is $\mathcal{P}$-Lipschitz-continuous, if there is an $L>0$, such that for all $\epsilon>0$ and every common distribution of $X$ and $Y$ from Definition \ref{def1}:
\begin{equation*}
 P\left[t-\epsilon\leq\left|X-Y\right|\leq t+\epsilon\right]\leq L\epsilon
\end{equation*}
The difference between $\mathds{1}_{\left\{\left|X-Y\right|<t\right\}}$ and $\mathds{1}_{\left\{\left|X'-Y\right|<t\right\}}$ is not 0, iff $\left|X-Y\right|<t$ and $\left|X'-Y\right|\geq t$ or the other way round. As $\left|X-Y'\right|\leq\epsilon$, it follows that $t-\epsilon\leq\left|X-Y\right|\leq t+\epsilon$. Therefore
\begin{equation*}
E\left[\left|\mathds{1}_{\left\{\left|X-Y\right|<t\right\}}-\mathds{1}_{\left\{\left|X'-Y\right|<t\right\}}\right|\mathds{1}_{\left\{\left|X-X'\right|\leq\epsilon\right\}}\right]\leq P\left[t-\epsilon\leq\left|X-Y\right|\leq t+\epsilon\right]\leq L\epsilon.
\end{equation*}
\end{exa}

\begin{exa}[$\mathcal{P}$-Lipschitz-discontinuity] Consider the kernel $h\left(x,y\right)=\mathds{1}_{\left\{x\geq 0\right\}}+\mathds{1}_{\left\{y\geq 0\right\}}$ and let the $X_i$ have the density $f\left(t\right)=\frac{1}{6}\left|t\right|^{-\frac{2}{3}}\mathds{1}_{\left[-1,1\right]\setminus\left\{0\right\}}\left(t\right)$. Then for independent random variables $X$, $X'$ and $Y$ with density $f$

\begin{multline*}
 E\left[\left|h\left(X,Y\right)-h\left(X',Y\right)\right|\mathds{1}_{\left\{\left|X-X'\right|\leq\epsilon\right\}}\right]\geq P\left[X\in\left[0,\frac{\epsilon}{2}\right]\right]P\left[X'\in\left[-\frac{\epsilon}{2},0\right) \right]\\
=4^{-\frac{4}{3}}\epsilon^{\frac{2}{3}}.
\end{multline*}
So this kernel $h$ is not $\mathcal{P}$-Lipschitz-continuous, because the probability distribution is concentrated in the neighborhood of the jump of $h$.
\end{exa}

It becomes clear from the examples that it depends not only on the kernel $h$, but also on the distribution $\mathcal{P}$, whether the kernel $h$ is $\mathcal{P}$-Lipschitz-continuous. We extend the CLT for U-statistics to strongly mixing data using the Hoeffding-decomposition and a new generalized covariance inequality. The strong mixing assumption is weaker than absolute regularity (as in Yoshihara's CLT), but this comes with the price of more technical conditions: A faster decay of mixing coefficients, some finite moments of $X_i$ and the additional $\mathcal{P}$-Lipschitz-continuity of the kernel.

\begin{prop}\label{prop4}Let $\left(X_{n}\right)_{n\in\mathbb{N}}$ be a stationary, mixing process and $h$ a kernel, such that for a $\delta>0$, $M>0$:
\begin{align*}
\iint\left|h\left(x_{1},x_{2}\right)\right|^{2+\delta}dF\left(x_{1}\right)dF\left(x_{2}\right)&\leq M\\
\forall k\in\mathbb{N}_{0}:\quad\int\left|h\left(x_{1},x_{1+k}\right)\right|^{2+\delta}dP\left(x_{1},x_{1+k}\right)&\leq M
\end{align*}
If one of the following two conditions holds
\begin{itemize}
 \item for a $\delta'\in\left(0,\delta\right)$:  $\beta\left(n\right)=O\left(n^{-\frac{2+\delta'}{\delta'}}\right)$
 \item $h$ is $\mathcal{P}$-Lipschitz-continuous, $E\left|X_{1}\right|^{\gamma}<\infty$ for a $\gamma>0$ and for $\rho>\frac{3\gamma\delta+\delta+5\gamma+2}{2\gamma\delta}$: $\alpha\left(n\right)=O\left(n^{-\rho}\right)$
\end{itemize}
then
\begin{equation}
\sqrt{n}\left(U_{n}\left(h\right)-\theta\right)\xrightarrow{\mathcal{D}}N\left(0,4\sigma_{\infty}^{2}\right)
\end{equation}
with $\sigma_{\infty}^{2}=\var\left[h_{1}\left(X_{1}\right)\right]+2\sum_{k=1}^{\infty}\cov\left[h_{1}\left(X_{1}\right)h_{1}\left(X_{1+k}\right)\right]$.
\end{prop}

\section{Bootstrap for U-statistics}

There is a variety of block bootstrap methods (see Lahiri \cite{lah2}), we consider the circular block bootstrap introduced by Politis and Romano \cite{poli}. Instead of the original sample of n observations with an unknown distribution, construct new samples $X^\star_{1},\ldots,X^\star_{bl}$ as follows: Extend the sample $X_{1},\ldots,X_{n}$ periodically by $X_{i+n}=X_{i}$, choose blocks of $l=l_n$ consecutive observations of the sample randomly and repeat that $b=\lfloor\frac{n}{l}\rfloor$ times independently: For $j=1,\ldots,n$, $k=0,\ldots,b-1$
\begin{equation*}
P^\star\left[X^\star_{kl+1}=X_{j},\ldots,X^\star_{(k+1)l}=X_{j+l-1}\right]=\frac{1}{n},
\end{equation*}
where $P^{\star}$ is the bootstrap distribution conditionally on $\left(X_n\right)_{n\in\mathbb{N}}$, $E^{\star}$ and $\operatorname{Var}^{\star}$ are the conditional expectation and variance. Note that $E^{\star}\left[X_i^{\star}\right]=\frac{1}{n}\sum_{i=1}^{n}X_{i}=\bar{X}$. For strong mixing stationary processes, Shao and Yu \cite{shao} proved that the bootstrap version of the sample mean $\bar{X}^{\star}_{n}=\frac{1}{bl}\sum_{i=1}^{bl}X_{i}^{\star}$ has almost surely the same asymptotic distribution as the sample mean $\bar{X}$ and that the variance of $\bar{X}^{\star}_{n}$ and of $\bar{X}$ converge to the same limit.

With increasing block length $l$, the bias of the bootstrap variance estimator $\var^{\star}\left[\sqrt{bl}\bar{X}^{\star}_{n}\right]$ becomes smaller and the variance becomes bigger. By minimizing the mean squared error ($\operatorname{MSE}$) of $\var^{\star}\left[\sqrt{bl}\bar{X}^{\star}_{n}\right]$, one gets the following rate of convergence (see Lahiri \cite{lahi}):

\begin{equation*}
\min_{\substack{l\\l^{-1}+l^{2}n^{-1}\rightarrow0}}\operatorname{MSE}\left(\var^{\star}\left[\sqrt{bl}\bar{X}^{\star}_{n}\right]\right)=O\left(n^{-\frac{2}{3}}\right).
\end{equation*}

Naik-Nimbalkar and Rajarshi \cite{naik} have shown that the consistency of the block bootstrap holds also for the empirical process. Furthermore, the block bootstrap is valid for smooth functions of means and differentiable functionals of the empirical process (e.g. L-statistics), as well as for M-estimators; see the book of Lahiri \cite{lah2}, chapter 4.

The bootstrap for U-statistics has so far only been studied in the independent case, beginning with Bickle and Freedman \cite{bick}, and extended to degenerate U-statistics by Arcones, Gin\'{e} \cite{arco} and Dehling, Mikosch \cite{deh2}, to studentized U-statistics by Helmers \cite{helm} and to weighted bootstrap by Janssen \cite{jans}. 

To bootstrap U-statistics from times series, one can resample blocks of observations and plug them in:

\begin{multline*}
 U^{\star}_{n}\left(h\right)=\frac{2}{bl\left(bl-1\right)}\sum_{1\leq i<j\leq bl}h\left(X^{\star}_{i},X^{\star}_{j}\right)\\
=\theta+\frac{2}{bl}\sum_{i=1}^{bl}h_{1}\left(X_{i}^{\star}\right)+\frac{2}{bl\left(bl-1\right)}\sum_{1\leq i<j\leq bl}h_{2}\left(X^{\star}_{i},X^{\star}_{j}\right)
\end{multline*}

We show that for strongly mixing data the circular block bootstrap version of a U-statistic has the same asymptotic variance and the same normal limit distribution as the U-statistic itself.

\begin{prop}\label{prop3}Let $\left(X_{n}\right)_{n\in\mathbb{N}}$ be a stationary, mixing process and $h$ a kernel, such that for a $\delta>0$, $M>0$:
\begin{align*}
\iint\left|h\left(x_{1},x_{2}\right)\right|^{2+\delta}dF\left(x_{1}\right)dF\left(x_{2}\right)&\leq M\\
\forall k\in\mathbb{N}_{0}:\quad\int\left|h\left(x_{1},x_{1+k}\right)\right|^{2+\delta}dP\left(x_{1},x_{1+k}\right)&\leq M
\end{align*}
Let $l$ be the block length with $l\xrightarrow{n\rightarrow\infty}\infty$ and $l=O\left(n^{1-\epsilon}\right)$ for some $\epsilon>0$. If one of the following two conditions holds
\begin{itemize}
 \item for a $\delta'\in\left(0,\delta\right)$:  $\beta\left(n\right)=O\left(n^{-\frac{2+\delta'}{\delta'}}\right)$
 \item $h$ is $\mathcal{P}$-Lipschitz-continuous, $E\left|X_{1}\right|^{\gamma}<\infty$ for a $\gamma>0$ and for $\rho>\frac{3\gamma\delta+\delta+5\gamma+2}{2\gamma\delta}$: $\alpha\left(n\right)=O\left(n^{-\rho}\right)$
\end{itemize}
then
\begin{align}
\left|\var^{\star}\left[\sqrt{bl}U^{\star}_{n}\left(h\right)\right]-\var\left[\sqrt{n}U_{n}\left(h\right)\right]\right|&\xrightarrow{\mathcal{P}}0\label{line1}\\
\sup_{x\in\mathbb{R}}\left|P^{\star}\left[\sqrt{bl}\left(U^{\star}_{n}\left(h\right)-E^{\star}\left[U^{\star}_{n}\right]\right)\leq x\right]-P\left[\sqrt{n}\left(U_{n}\left(h\right)-\theta\right)\leq x\right]\right|&\xrightarrow{\mathcal{P}}0\label{line2}.
\end{align}
\end{prop}

If we assume the existence of higher moments, we can achieve almost sure convergence:
\begin{prop}\label{prop5}Let $\left(X_{n}\right)_{n\in\mathbb{N}}$ be a stationary and absolutely regular process and $h$ a kernel, such that for a $\delta>0$, $M>0$:
\begin{align*}
\iint\left|h\left(x_{1},x_{2}\right)\right|^{4+\delta}dF\left(x_{1}\right)dF\left(x_{2}\right)&\leq M\\
\forall k\in\mathbb{N}_{0}:\quad\int\left|h\left(x_{1},x_{1+k}\right)\right|^{4+\delta}dP\left(x_{1},x_{1+k}\right)&\leq M
\end{align*}
and for a $\delta'\in\left(0,\delta\right)$ $\beta\left(n\right)=O\left(n^{-\frac{3\left(4+\delta'\right)}{\delta'}}\right)$ and additionally $l\xrightarrow{n\rightarrow\infty}\infty$ and $l=O\left(n^{1-\epsilon}\right)$ for some $\epsilon>0$, then 
\begin{align}
\left|\var^{\star}\left[\sqrt{bl}U^{\star}_{n}\left(h\right)\right]-\var\left[\sqrt{n}U_{n}\left(h\right)\right]\right|&\xrightarrow{a.s.}0\\
\sup_{x\in\mathbb{R}}\left|P^{\star}\left[\sqrt{bl}\left(U^{\star}_{n}\left(h\right)-E^{\star}\left[U^{\star}_{n}\right]\right)\leq x\right]-P\left[\sqrt{n}\left(U_{n}\left(h\right)-\theta\right)\leq x\right]\right|&\xrightarrow{a.s.}0.
\end{align}
\end{prop}

The degenerate part of the bootstrapped U-statistic converges to zero with a rate, which does not depend on the block length and is faster than the convergence of the sample mean. Choosing the optimal block length for the block bootstrap variance estimator of the linear part $\frac{2}{\sqrt{n}}\sum_{i=1}^{n}h_{1}\left(X_{i}\right)$, we can achieve the following rate of convergence:

\begin{cor}\label{cor1}
Let $\left(X_{n}\right)_{n\in\mathbb{N}}$ be a stationary and absolutely regular process and $h$ a kernel, such that for a $\delta>0$, $M>0$:
\begin{align*}
\iint\left|h\left(x_{1},x_{2}\right)\right|^{6+\delta}dF\left(x_{1}\right)dF\left(x_{2}\right)&\leq M\\
\forall k\in\mathbb{N}_{0}:\quad\int\left|h\left(x_{1},x_{1+k}\right)\right|^{4+\frac{2}{3}\delta}dP\left(x_{1},x_{1+k}\right)&\leq M
\end{align*}
and for a $\delta'\in\left(0,\delta\right)$ $\beta\left(n\right)=O\left(n^{-\frac{3\left(6+\delta'\right)}{\delta'}}\right)$, the variance estimator converges with the following rate:
\begin{equation}
\min_{\substack{l\\l^{-1}+l^{2}n^{-1}\rightarrow0}}\operatorname{MSE}\left(\var^{\star}\left[\sqrt{bl}U^{\star}_{n}\left(h\right)\right]\right)=O\left(n^{-\frac{2}{3}}\right)
\end{equation}
\end{cor}

\begin{rem}
If
\begin{align*}
\var\left[h_1\left(X_1\right)\right]+2\sum_{k\geq2}\cov\left[h_1\left(X_{1}\right),h_1\left(X_{k}\right)\right]&>0\\
\sum_{k\geq1}k\cov\left[h_1\left(X_{1}\right),h_1\left(X_{1+k}\right)\right]&\neq0,
   \end{align*}
then the optimal block length $l^0=\operatorname{argmin}\left(\operatorname{MSE}\left(\var^{\star}\left[\sqrt{bl}U^{\star}_{n}\left(h\right)\right]\right)\right)$ has the form $l^0=Kn^{-\frac{1}{3}}+o\left(n^{-\frac{1}{3}}\right)$ for a constant $K$ (see Corollary 3.1 of Lahiri \cite{lahi}). To find this block length $l^0$, one can use the following subsampling method introduced by Hall, Horowitz and Jing \cite{hall}:

Choose a pilot block sitzen $l_n^\star$ and a subsampling size $m=m_n$ such that $m^{-1}+mn^{-1}\rightarrow0$ and minimize
\begin{equation*}
 \widehat{\operatorname{MSE}}\left(l\right)=\frac{1}{n-m+1}\sum_{k=1}^{n-m+1}\left(\var^{\star}_{l}\left[\sqrt{m}U_{m,k}^{\star}\left(h\right)\right]-\var^{\star}_{l_n^{\star}}\left[\sqrt{n}U_{n}^{\star}\left(h\right)\right]\right)^{2},
\end{equation*}
where $\var^{\star}_{l}$ is the bootstrap variance if the block length is $l$ and
\begin{equation*}
U_{m,k}^{\star}\left(h\right)=\frac{2}{m(m-1)}\sum_{k\leq i<j\leq k+m-1}h\left(X_i^{\star},X_j^{\star}\right)
\end{equation*}
is the bootstrapped U-statistic of the $m$ observations starting with $X_k^{\star}$. Choose a small $\epsilon>0$ and set
\begin{equation*}
 \hat{l}^0=\left(\frac{n}{m}\right)^{\frac{1}{3}}\operatorname{argmin}_{\epsilon m^\frac{1}{3}\leq l\leq \frac{1}{\epsilon} m^\frac{1}{3} }\left(\widehat{\operatorname{MSE}}\left(l\right)\right)
\end{equation*}
as the estimated optimal block length. The consistency of this subsampling method has been proved by Nordman, Lahiri and Fridley \cite{nord} for the sample mean.
\end{rem}

\begin{rem} Theorem \ref{prop3} and Corollary \ref{cor1} hold not only for the circular block bootstrap, but also for the moving block and the non overlapping block bootstrap. For a proof, note first that there are results analogous to the theorem of Shao, Yu \cite{shao} (see the book of Lahiri \cite{lah2} and the references therein) for these bootstrapping methods. Theorem 3.3 of Lahiri \cite{lahi} treats all three methods. Moreover, the bounds for the bootstrap version of the degenerate part $U_n\left(h_2\right)$ (Lemma \ref{lem4} and \ref{lem5}) remain valid.
\end{rem}

\medskip
\noindent
\textbf{Simulation results:} We study the estimator for the variance $\sigma^2=\var\left[X_i\right]$, which can be expressed as a U-statistic (see Example \ref{ex1})
\begin{equation*}
\hat{\sigma}^{2}=\frac{1}{n-1}\sum_{i=1}^{n}\left(X_{i}-\bar{X}\right)^{2}=\frac{n}{n-1}\left(\overline{X^{2}}-\bar{X}^{2}\right)
\end{equation*}
and the stationary autoregressive process defined by $X_n=\frac{1}{2}X_{n-1}+\epsilon_n$, where $\left(\epsilon_n\right)_{n\in\mathbb{N}}$ is a sequence of iid standard normal random variables. The distance between the real and the bootstrapped distribution function
\begin{equation*}
D_{boot}=\sup_{x\in\mathbb{R}}\left|P^{\star}\left[\sqrt{bl}\left(\hat{\sigma}^{2\star}-E^{\star}\left[\hat{\sigma}^{2\star}\right]\right)\leq x\right]-P\left[\sqrt{n}\left(\hat{\sigma}^{2}-\sigma^{2}\right)\leq x\right]\right|
\end{equation*}
is compared to
\begin{equation*}
D_{norm}=\sup_{x\in\mathbb{R}}\left|\Phi\left(\frac{x}{\sqrt{n\widehat{\var[\hat{\sigma}^{2}]}}}\right)-P\left[\sqrt{n}\left(\hat{\sigma}^{2}-\sigma^{2}\right)\leq x\right]\right|,
\end{equation*}
where $\Phi$ is the distribution function of a standard normal random variable. The covariance matrix of $\left(\bar{X},\overline{X^{2}}\right)^t$ is estimated using the moment method, including the autocovariances for lags not bigger than $l$. Applying the $\delta$-method, one obtains:
\begin{multline*}
\widehat{\var[\hat{\sigma}^{2}]}=\frac{1}{(n-1)^{2}}\left(\sum_{\substack{i,j\\ \left|i-j\right|\leq l}}\left(X^{2}_{i}-\overline{X^{2}}\right)\left(X^{2}_{j}-\overline{X^{2}}\right)\right.\\
\left.-4\bar{X}\sum_{\substack{i,j\\ \left|i-j\right|\leq l}}\left(X^{2}_{i}-\overline{X^{2}}\right)\left(X_{j}-\bar{X}\right)+4\bar{X}^{2}\sum_{\substack{i,j\\ \left|i-j\right|\leq l}}\left(X_i-\bar{X}\right)\left(X_{j}-\bar{X}\right)\right)
\end{multline*}
We have calculated the distances $D_{boot}$ and $D_{norm}$ with the empirical distribution function of 10,000 random variables.

The following table shows the mean of 1,000 realizations of $D_{boot}$ and $D_{norm}$ for different sample sizes $n$ and block lengths $l$, where the block lengths are integer approximations to $n^{\frac{1}{3}}$.  In all cases, the moving block bootstrap performs better than the normal approximation:
\begin{center}
\begin{tabular}{l|l||c|c}
sample size $n$&block length $l$  &bootstrap  &normal approx.\\ 
\hline\hline 24 & 3 & 0.153 & 0.196 \\ 
48 & 4 & 0.111 & 0.125 \\ 
100 & 5 & 0.076 & 0.091 \\ 
200 & 6 & 0.060 & 0.073 \\ 
500 & 8 & 0.039 & 0.046
\end{tabular}
\end{center}

The boxplots below give a closer look at the distributions of $D_{boot}$ and $D_{norm}$. The bootstrap version $D_{boot}$ has not only the lower median, but produces far less outliers than the normal approximation.

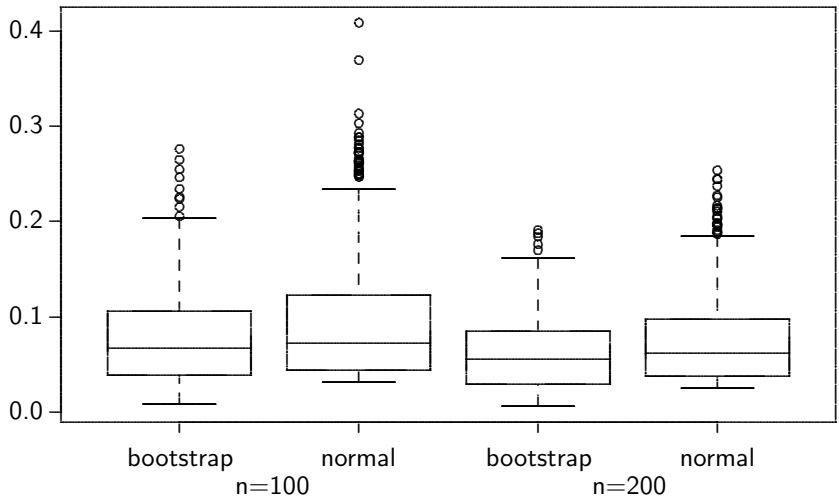
\hbox{\beginpicture
\setcoordinatesystem units <1pt,1pt>
\setplotarea x from 0 to 361.35, y from 0 to 216.81
\setlinear
\font\picfont cmss10\picfont
\font\picfont cmss10 at 10pt\picfont
\font\picfont cmss10 at 10pt\picfont
\setdashpattern <15pt, 15pt, 15pt, 15pt, 15pt, 15pt, 15pt, 15pt>
\plot 51.80 60.23 106.06 60.23 /
\plot 106.06 60.23 106.06 84.10 /
\plot 106.06 84.10 51.80 84.10 /
\plot 51.80 84.10 51.80 60.23 /
\setsolid
\plot 51.80 70.20 106.06 70.20 /
\setdashpattern <4pt, 4pt>
\plot 78.93 49.14 78.93 60.23 /
\setdashpattern <4pt, 4pt>
\plot 78.93 119.72 78.93 84.10 /
\setsolid
\plot 65.37 49.14 92.50 49.14 /
\setsolid
\plot 65.37 119.72 92.50 119.72 /
\setsolid
\plot 65.37 49.14 92.50 49.14 /
\setsolid
\plot 65.37 119.72 92.50 119.72 /
\setsolid
\plot 65.37 49.14 92.50 49.14 /
\setsolid
\plot 65.37 119.72 92.50 119.72 /
\setsolid
\plot 65.37 49.14 92.50 49.14 /
\setsolid
\plot 65.37 119.72 92.50 119.72 /
\setsolid
\plot 51.80 60.23 106.06 60.23 /
\plot 106.06 60.23 106.06 84.10 /
\plot 106.06 84.10 51.80 84.10 /
\plot 51.80 84.10 51.80 60.23 /
\circulararc 360 degrees from 78.93 129.09 center at 78.93 127.53
\circulararc 360 degrees from 78.93 147.31 center at 78.93 145.75
\circulararc 360 degrees from 78.93 136.51 center at 78.93 134.95
\circulararc 360 degrees from 78.93 143.21 center at 78.93 141.65
\circulararc 360 degrees from 78.93 121.78 center at 78.93 120.22
\circulararc 360 degrees from 78.93 139.50 center at 78.93 137.94
\circulararc 360 degrees from 78.93 128.45 center at 78.93 126.88
\circulararc 360 degrees from 78.93 132.15 center at 78.93 130.59
\circulararc 360 degrees from 78.93 125.38 center at 78.93 123.82
\setdashpattern <15pt, 15pt, 15pt, 15pt, 15pt, 15pt, 15pt, 15pt>
\plot 119.63 62.08 173.89 62.08 /
\plot 173.89 62.08 173.89 90.44 /
\plot 173.89 90.44 119.63 90.44 /
\plot 119.63 90.44 119.63 62.08 /
\setsolid
\plot 119.63 72.24 173.89 72.24 /
\setdashpattern <4pt, 4pt>
\plot 146.76 57.35 146.76 62.08 /
\setdashpattern <4pt, 4pt>
\plot 146.76 130.41 146.76 90.44 /
\setsolid
\plot 133.20 57.35 160.33 57.35 /
\setsolid
\plot 133.20 130.41 160.33 130.41 /
\setsolid
\plot 133.20 57.35 160.33 57.35 /
\setsolid
\plot 133.20 130.41 160.33 130.41 /
\setsolid
\plot 133.20 57.35 160.33 57.35 /
\setsolid
\plot 133.20 130.41 160.33 130.41 /
\setsolid
\plot 133.20 57.35 160.33 57.35 /
\setsolid
\plot 133.20 130.41 160.33 130.41 /
\setsolid
\plot 119.63 62.08 173.89 62.08 /
\plot 173.89 62.08 173.89 90.44 /
\plot 173.89 90.44 119.63 90.44 /
\plot 119.63 90.44 119.63 62.08 /
\circulararc 360 degrees from 146.76 143.07 center at 146.76 141.50
\circulararc 360 degrees from 146.76 149.15 center at 146.76 147.59
\circulararc 360 degrees from 146.76 137.27 center at 146.76 135.71
\circulararc 360 degrees from 146.76 141.95 center at 146.76 140.39
\circulararc 360 degrees from 146.76 147.42 center at 146.76 145.86
\circulararc 360 degrees from 146.76 160.64 center at 146.76 159.08
\circulararc 360 degrees from 146.76 151.74 center at 146.76 150.18
\circulararc 360 degrees from 146.76 139.18 center at 146.76 137.61
\circulararc 360 degrees from 146.76 138.82 center at 146.76 137.25
\circulararc 360 degrees from 146.76 146.13 center at 146.76 144.56
\circulararc 360 degrees from 146.76 150.41 center at 146.76 148.85
\circulararc 360 degrees from 146.76 136.80 center at 146.76 135.24
\circulararc 360 degrees from 146.76 195.06 center at 146.76 193.50
\circulararc 360 degrees from 146.76 139.75 center at 146.76 138.19
\circulararc 360 degrees from 146.76 137.66 center at 146.76 136.10
\circulararc 360 degrees from 146.76 136.80 center at 146.76 135.24
\circulararc 360 degrees from 146.76 143.57 center at 146.76 142.01
\circulararc 360 degrees from 146.76 142.71 center at 146.76 141.14
\circulararc 360 degrees from 146.76 157.04 center at 146.76 155.48
\circulararc 360 degrees from 146.76 145.80 center at 146.76 144.24
\circulararc 360 degrees from 146.76 153.36 center at 146.76 151.80
\circulararc 360 degrees from 146.76 151.71 center at 146.76 150.15
\circulararc 360 degrees from 146.76 147.85 center at 146.76 146.29
\circulararc 360 degrees from 146.76 146.02 center at 146.76 144.46
\circulararc 360 degrees from 146.76 137.92 center at 146.76 136.35
\circulararc 360 degrees from 146.76 145.95 center at 146.76 144.38
\circulararc 360 degrees from 146.76 141.19 center at 146.76 139.63
\circulararc 360 degrees from 146.76 140.22 center at 146.76 138.66
\circulararc 360 degrees from 146.76 180.98 center at 146.76 179.42
\circulararc 360 degrees from 146.76 142.49 center at 146.76 140.93
\circulararc 360 degrees from 146.76 142.13 center at 146.76 140.57
\circulararc 360 degrees from 146.76 145.80 center at 146.76 144.24
\setdashpattern <15pt, 15pt, 15pt, 15pt, 15pt, 15pt, 15pt, 15pt>
\plot 187.46 56.59 241.72 56.59 /
\plot 241.72 56.59 241.72 76.85 /
\plot 241.72 76.85 187.46 76.85 /
\plot 187.46 76.85 187.46 56.59 /
\setsolid
\plot 187.46 65.99 241.72 65.99 /
\setdashpattern <4pt, 4pt>
\plot 214.59 48.31 214.59 56.59 /
\setdashpattern <4pt, 4pt>
\plot 214.59 104.56 214.59 76.85 /
\setsolid
\plot 201.02 48.31 228.15 48.31 /
\setsolid
\plot 201.02 104.56 228.15 104.56 /
\setsolid
\plot 201.02 48.31 228.15 48.31 /
\setsolid
\plot 201.02 104.56 228.15 104.56 /
\setsolid
\plot 201.02 48.31 228.15 48.31 /
\setsolid
\plot 201.02 104.56 228.15 104.56 /
\setsolid
\plot 201.02 48.31 228.15 48.31 /
\setsolid
\plot 201.02 104.56 228.15 104.56 /
\setsolid
\plot 187.46 56.59 241.72 56.59 /
\plot 241.72 56.59 241.72 76.85 /
\plot 241.72 76.85 187.46 76.85 /
\plot 187.46 76.85 187.46 56.59 /
\circulararc 360 degrees from 214.59 108.96 center at 214.59 107.40
\circulararc 360 degrees from 214.59 116.53 center at 214.59 114.96
\circulararc 360 degrees from 214.59 114.08 center at 214.59 112.51
\circulararc 360 degrees from 214.59 115.34 center at 214.59 113.77
\circulararc 360 degrees from 214.59 111.20 center at 214.59 109.63
\setdashpattern <15pt, 15pt, 15pt, 15pt, 15pt, 15pt, 15pt, 15pt>
\plot 255.29 59.83 309.55 59.83 /
\plot 309.55 59.83 309.55 81.19 /
\plot 309.55 81.19 255.29 81.19 /
\plot 255.29 81.19 255.29 59.83 /
\setsolid
\plot 255.29 68.35 309.55 68.35 /
\setdashpattern <4pt, 4pt>
\plot 282.42 55.19 282.42 59.83 /
\setdashpattern <4pt, 4pt>
\plot 282.42 112.87 282.42 81.19 /
\setsolid
\plot 268.85 55.19 295.98 55.19 /
\setsolid
\plot 268.85 112.87 295.98 112.87 /
\setsolid
\plot 268.85 55.19 295.98 55.19 /
\setsolid
\plot 268.85 112.87 295.98 112.87 /
\setsolid
\plot 268.85 55.19 295.98 55.19 /
\setsolid
\plot 268.85 112.87 295.98 112.87 /
\setsolid
\plot 268.85 55.19 295.98 55.19 /
\setsolid
\plot 268.85 112.87 295.98 112.87 /
\setsolid
\plot 255.29 59.83 309.55 59.83 /
\plot 309.55 59.83 309.55 81.19 /
\plot 309.55 81.19 255.29 81.19 /
\plot 255.29 81.19 255.29 59.83 /
\circulararc 360 degrees from 282.42 125.85 center at 282.42 124.29
\circulararc 360 degrees from 282.42 120.96 center at 282.42 119.39
\circulararc 360 degrees from 282.42 139.18 center at 282.42 137.61
\circulararc 360 degrees from 282.42 117.79 center at 282.42 116.22
\circulararc 360 degrees from 282.42 133.27 center at 282.42 131.71
\circulararc 360 degrees from 282.42 135.79 center at 282.42 134.23
\circulararc 360 degrees from 282.42 115.84 center at 282.42 114.28
\circulararc 360 degrees from 282.42 124.52 center at 282.42 122.96
\circulararc 360 degrees from 282.42 118.97 center at 282.42 117.41
\circulararc 360 degrees from 282.42 121.24 center at 282.42 119.68
\circulararc 360 degrees from 282.42 123.69 center at 282.42 122.13
\circulararc 360 degrees from 282.42 125.02 center at 282.42 123.46
\circulararc 360 degrees from 282.42 118.54 center at 282.42 116.98
\circulararc 360 degrees from 282.42 115.55 center at 282.42 113.99
\circulararc 360 degrees from 282.42 115.12 center at 282.42 113.56
\circulararc 360 degrees from 282.42 124.88 center at 282.42 123.32
\circulararc 360 degrees from 282.42 129.63 center at 282.42 128.07
\circulararc 360 degrees from 282.42 128.99 center at 282.42 127.42
\circulararc 360 degrees from 282.42 116.53 center at 282.42 114.96
\circulararc 360 degrees from 282.42 135.83 center at 282.42 134.26
\circulararc 360 degrees from 282.42 118.97 center at 282.42 117.41
\circulararc 360 degrees from 282.42 121.96 center at 282.42 120.40
\circulararc 360 degrees from 282.42 118.72 center at 282.42 117.16
\setsolid
\plot 78.93 42.50 282.42 42.50 /
\setsolid
\plot 78.93 42.50 78.93 38.33 /
\setsolid
\plot 146.76 42.50 146.76 38.33 /
\setsolid
\plot 214.59 42.50 214.59 38.33 /
\setsolid
\plot 282.42 42.50 282.42 38.33 /
\put {bootstrap}  [lB] <0.00pt,0.00pt> at 59.13 25.83
\put {normal}  [lB] <0.00pt,0.00pt> at 132.40 25.83
\put {bootstrap}  [lB] <0.00pt,0.00pt> at 194.78 25.83
\put {normal}  [lB] <0.00pt,0.00pt> at 268.06 25.83
\put{n=100} [lB] <0.00pt,0.00pt> at 100.00 15.00
\put{n=200} [lB] <0.00pt,0.00pt> at 235.00 15.0
\setsolid
\plot 34.17 46.26 34.17 190.30 /
\setsolid
\plot 34.17 46.26 30.00 46.26 /
\setsolid
\plot 34.17 82.27 30.00 82.27 /
\setsolid
\plot 34.17 118.28 30.00 118.28 /
\setsolid
\plot 34.17 154.29 30.00 154.29 /
\setsolid
\plot 34.17 190.30 30.00 190.30 /
\put {0.0} at 20.83 46.87
\put {0.1} at 20.83 82.88
\put {0.2} at 20.83 118.89
\put {0.3} at 20.83 154.90
\put {0.4} at 20.83 190.91
\setsolid
\plot 34.17 42.50 327.18 42.50 /
\plot 327.18 42.50 327.18 199.31 /
\plot 327.18 199.31 34.17 199.31 /
\plot 34.17 199.31 34.17 42.50 /
\endpicture
}

\section{Auxiliary Results}
\subsection{Generalized Covariance Inequalities}

Yoshihara has proved the asymptotic normality of the U-statistic $U_{n}\left(h\right)$ with the help of the Hoeffding-decomposition and the following generalized covariance inequality:
\begin{lem}[Yoshihara \cite{yosh}]\label{lem1}
If there are $\delta,M>0$, so that for all $k\in\mathbb{N}_{0}$
\begin{align*}
\iint\left|h\left(x_{1},x_{2}\right)\right|^{2+\delta}dF\left(x_{1}\right)dF\left(x_{2}\right)&\leq M\\
\int\left|h\left(x_{1},x_{k}\right)\right|^{2+\delta}dP\left(x_{1},x_{k}\right)&\leq M
\end{align*}
then there is a constant $K$, such that for
 $m=\max\left\{i_{\left(2\right)}-i_{\left(1\right)},i_{\left(4\right)}-i_{\left(3\right)}\right\}$, where $i_{(1)}\leq i_{(2)}\leq i_{(3)}\leq i_{(4)}$ the following inequality holds:
\begin{equation}
\left|E\left[h_{2}\left(X_{i_{1}},X_{i_{2}}\right)h_{2}\left(X_{i_{3}},X_{i_{4}}\right)\right]\right|\leq K\beta^{\frac{\delta}{2+\delta}}\left(m\right)
\end{equation}
\end{lem}

To prove Lemma \ref{lem1} under absolute regularity, one can use coupling techniques (see Berbee \cite{berb} and Berkes, Philipp \cite{berk}): For dependent random variables $X$ and $Y$, one can find a random variable $X'$, such that
\begin{itemize}
 \item $X'$ has the same distribution as $X$,
 \item $X'$ and $Y$ are independent,
 \item $P\left[X\neq X'\right]=\beta\left(X,Y\right)$.
\end{itemize}
Such a coupling is impossible under strong mixing, as can be seen e.g. from the results of Dehling \cite{dehl}. Bradley \cite{brad}, however, was able to establish a weaker type of coupling for strong mixing random variables, using the fact that absolute regularity and strongly mixing are equivalent for random variables taking their values in a finite set and approximating general random variables by such discrete ones:
\begin{lem}[Bradley \cite{brad}]\label{lem6} Let $X$, $Y$ be random variables, $X$ real-valued with $E\left|X\right|^{\gamma}\leq\infty$. Let $0<\epsilon\leq\left\|X\right\|_{\gamma}$. Then there exists (after replacing the underlying probability space by a bigger one if necessary) a random variable $X'$ such that
\begin{itemize}
 \item $X'$ has the same distribution as $X$,
 \item $X'$ and $Y$ are independent,
 \item\begin{equation}
 P\left[\left|X-X'\right|\geq\epsilon\right]\leq18\frac{\left\|X\right\|_{\gamma}^{\frac{\gamma}{2+\gamma}}}{\epsilon^{\frac{\gamma}{2+\gamma}}}\alpha^{\frac{2\gamma}{2+\gamma}}\left(X,Y\right).
\end{equation}
\end{itemize}
\end{lem}
As this coupling under strong mixing allows small differences between $X$ and $X'$ (while $X$ and $X'$ are equal with high probability in the case of absolute regularity), we need the $\mathcal{P}$-Lipschitz-continuity of the kernel.

\begin{lem}\label{lem8} Let $h$ be a $\mathcal{P}$-Lipschitz-continuous kernel with constant $L$, $\left(X_{n}\right)_{n\in\mathbb{N}}$ a stationary sequence of random variables. If there is a $\gamma>0$ with $E\left|X_{k}\right|^{\gamma}<\infty$ and $M>0$, $\delta>0$, so that for all $k\in\mathbb{N}_{0}$
\begin{align*}
\iint\left|h\left(x_{1},x_{2}\right)\right|^{2+\delta}dF\left(x_{1}\right)dF\left(x_{2}\right)&\leq M\\
\int\left|h\left(x_{1},x_{k}\right)\right|^{2+\delta}dP\left(x_{1},x_{k}\right)&\leq M
\end{align*}
then there exists a constant $K=K\left(\gamma,\left\|X_{1}\right\|_{\gamma},\delta,M,L\right)$, such that the following inequality holds with $m=\max\left\{i_{(2)}-i_{(1)},i_{(4)}-i_{(3)}\right\}$:
\begin{equation}
\left|E\left[h_{2}\left(X_{i_{1}},X_{i_{2}}\right)h_{2}\left(X_{i_{3}},X_{i_{4}}\right)\right]\right|\leq K\alpha^{\frac{2\gamma\delta}{3\gamma\delta+\delta+5\gamma+2}}\left(m\right)
\end{equation}
\end{lem}

\subsection{Bounds for the Degenerate Part of a U-Statistic}
With the covariance inequalities one can show that the covariance of the summands $h\left(X_{i},X_{j}\right)$ is small if the gap between the indices is big enough. Therefore, the degenerate part decreases fast enough, so that it does not disturb the asymptotic normality of $\frac{1}{\sqrt{n}}\sum h_{1}\left(X_{i}\right)$.

\begin{lem}[Yoshihara \cite{yosh}]\label{lem2} If the assumptions of Lemma \ref{lem1} hold and furthermore for a $\delta'<\delta$
\begin{equation*}
\beta\left(n\right)=O\left(n^{-\frac{2+\delta'}{\delta'}}\right)
\end{equation*}
then for $U_{n}\left(h_{2}\right)$:
\begin{multline}
E\left[nU_{n}^{2}\left(h_{2}\right)\right]\leq\frac{4}{n\left(n-1\right)^{2}}\sum_{1\leq i_{1}<i_{2}\leq n}\sum_{1\leq i_{3}<i_{4}\leq n}\left|E\left[h_{2}\left(X_{i_{1}},X_{i_{2}}\right)h_{2}\left(X_{i_{3}},X_{i_{4}}\right)\right]\right|\\
\leq\frac{4}{n^{3}}\sum_{i_{1},i_{2},i_{3},i_{4}=1}^{n}\left|E\left[h_{2}\left(X_{i_{1}},X_{i_{2}}\right)h_{2}\left(X_{i_{3}},X_{i_{4}}\right)\right]\right|=O\left(n^{-\eta}\right)
\end{multline}
with $\eta=\min\left\{2\frac{\delta-\delta'}{\delta'\left(2+\delta\right)},1\right\}$.
\end{lem}

So $\sqrt{n}U_{n}\left(h_{2}\right)$ vanishes as $n$ increases. For one of our later results, we also need another one of Yoshihara's lemmas (Our assumptions and result differ slightly from the lemma in \cite{yosh}, as we believe there is a misprint):

\begin{lem}[Yoshihara \cite{yosh}]\label{lem3} If
\begin{align*}
\iint\left|h\left(x_{1},x_{2}\right)\right|^{4+\delta}dF\left(x_{1}\right)dF\left(x_{2}\right)&\leq M\\
\forall k\in\mathbb{N}_{0}:\quad\int\left|h\left(x_{1},x_{1+k}\right)\right|^{4+\delta}dP\left(x_{1},x_{1+k}\right)&\leq M
\end{align*}
and for a $\delta'\in\left(0,\delta\right)$ $\beta\left(n\right)=O\left(n^{-\frac{3\left(4+\delta'\right)}{\delta'}}\right)$, then for $\eta'=\min\left\{12\frac{\delta-\delta'}{\delta'\left(4+\delta\right)},1\right\}$
\begin{multline}
E\left[n^{2}U_{n}^{4}\left(h_{2}\right)\right]\\
\leq\frac{16}{n^{6}}\sum_{i_{1},\ldots,i_{8}=1}^{n}\left|E\left[h_{2}\left(X_{i_{1}},X_{i_{2}}\right)h_{2}\left(X_{i_{3}},X_{i_{4}}\right)h_{2}\left(X_{i_{5}},X_{i_{6}}\right)h_{2}\left(X_{i_{7}},X_{i_{8}}\right)\right]\right|\\
=O\left(n^{-1-\eta'}\right).
\end{multline}
\end{lem}

Now we show a result analogous to the Lemma \ref{lem2} under strong mixing:

\begin{lem}\label{lem9} If the assumptions of Lemma \ref{lem8} hold and for a $\rho>\frac{3\gamma\delta+\delta+5\gamma+2}{2\gamma\delta}$
\begin{equation*}
\alpha\left(n\right)=O\left(n^{-\rho}\right)
\end{equation*}
then for $U_{n}\left(h_{2}\right)$:
\begin{multline}
E\left[nU_{n}^{2}\left(h_{2}\right)\right]\leq\frac{4}{n\left(n-1\right)^{2}}\sum_{1\leq i_{1}<i_{2}\leq n}\sum_{1\leq i_{3}<i_{4}\leq n}\left|E\left[h_{2}\left(X_{i_{1}},X_{i_{2}}\right)h_{2}\left(X_{i_{3}},X_{i_{4}}\right)\right]\right|\\
\leq\frac{4}{n^{3}}\sum_{i_{1},i_{2},i_{3},i_{4}=1}^{n}\left|E\left[h_{2}\left(X_{i_{1}},X_{i_{2}}\right)h_{2}\left(X_{i_{3}},X_{i_{4}}\right)\right]\right|=O\left(n^{-\eta}\right)
\end{multline}
with $\eta=\min\left\{\rho\frac{2\gamma\delta}{3\gamma\delta+\delta+5\gamma+2}-1,1\right\}>0$.
\end{lem}

We need a bound for $U_{n}^{\star}\left(h_{2}\right)=\frac{2}{bl\left(bl-1\right)}\sum_{1\leq i<j\leq n}h_{2}\left(X^{\star}_{i},X^{\star}_{j}\right)$. Using Yoshihara's inequality for the second moment respectively Lemma \ref{lem9} and using the fact that the bootstrap expectation of a U-statistic is similar to a von Mises-statistic, we get:

\begin{lem}\label{lem4}
Let $\left(X_{n}\right)_{n\in\mathbb{N}}$ be a stationary, mixing process and $h$ a kernel, such that for a $\delta>0$, $M>0$:
\begin{align*}
\iint\left|h\left(x_{1},x_{2}\right)\right|^{2+\delta}dF\left(x_{1}\right)dF\left(x_{2}\right)&\leq M\\
\forall k\in\mathbb{N}_{0}:\quad\int\left|h\left(x_{1},x_{1+k}\right)\right|^{2+\delta}dP\left(x_{1},x_{1+k}\right)&\leq M
\end{align*}
If one of the following two conditions holds
\begin{itemize}
 \item for a $\delta'\in\left(0,\delta\right)$: $\beta\left(n\right)=O\left(n^{-\frac{2+\delta'}{\delta'}}\right)$
 \item $h$ is $\mathcal{P}$-Lipschitz-continuous, $E\left|X_{1}\right|^{\gamma}<\infty$ for a $\gamma>0$ and for $\rho>\frac{3\gamma\delta+\delta+5\gamma+2}{2\gamma\delta}$: $\alpha\left(n\right)=O\left(n^{-\rho}\right)$
\end{itemize}
then for $\eta=\min\left\{2\frac{\delta-\delta'}{\delta'\left(2+\delta\right)},1\right\}$ respectively $\eta=\min\left\{\rho\frac{2\gamma\delta}{3\gamma\delta+\delta+5\gamma+2}-1,1\right\}$:
\begin{equation}
E\left[E^{\star}\left[blU_{n}^{\star 2}\left(h_{2}\right)\right]\right]=O\left(n^{-\eta}\right).
\end{equation}
\end{lem}

With the inequality for the fourth moment, we can calculate a faster rate of convergence. Note that this rate does not depend on the block length.

\begin{lem}\label{lem5} If
\begin{align*}
\iint\left|h\left(x_{1},x_{2}\right)\right|^{4+\delta}dF\left(x_{1}\right)dF\left(x_{2}\right)&\leq M\\
\forall k\in\mathbb{N}_{0}:\quad\int\left|h\left(x_{1},x_{1+k}\right)\right|^{4+\delta}dP\left(x_{1},x_{1+k}\right)&\leq M
\end{align*}
and for a $\delta'\in\left(0,\delta\right)$ $\beta\left(n\right)=O\left(n^{-\frac{3\left(4+\delta'\right)}{\delta'}}\right)$, then for $\eta'=\min\left\{12\frac{\delta-\delta'}{\delta'\left(4+\delta\right)},1\right\}$

\begin{equation}
E\left[E^{\star}\left[\left(bl\right)^{2}U_{n}^{\star 4}\left(h_{2}\right)\right]\right]=O\left(n^{-1-\eta'}\right).
\end{equation}
\end{lem}

\section{Proofs}
We will first prove the auxiliary results and after that the CLT and the theorems about the bootstrap.

\subsection{Auxiliary Results}

\begin{proof}[Proof of Lemma \ref{lem8}:]
For simplicity, we consider only the case $i_{1}<i_{2}<i_{3}<i_{4}$ and $i_{2}-i_{1}\geq i_{4}-i_{3}$. Let $\epsilon>0$, $K>0$ and define:
\begin{equation*}
h_{2,K}\left(x,y\right) = \left\{ \begin{array}{ll}
h_{2}\left(x,y\right) & \textrm{if $\left|h_{2}\left(x,y\right)\right|\leq\sqrt{K}$}\\
\sqrt{K} & \textrm{if $h_{2}\left(x,y\right)>\sqrt{K}$}\\
-\sqrt{K} & \textrm{if $h_{2}\left(x,y\right)<-\sqrt{K}$}
\end{array} \right.
\end{equation*}

$h_{2}$ is $\mathcal{P}$-Lipschitz-continuous with constant $2L$, as for all $X$, $X'$, $Y$ as in defintion \ref{def1} and $Y'$ with the same distribution as $Y$ and independent of $X$ and $X'$:
\begin{align*}
 &E\left[\left|h_{2}\left(X,Y\right)-h_{2}\left(X',Y\right)\right|\mathds{1}_{\left\{\left|X-X'\right|\leq\epsilon\right\}}\right]\\
\leq&E\left[\left|h\left(X,Y\right)-h\left(X',Y\right)\right|\mathds{1}_{\left\{\left|X-X'\right|\leq\epsilon\right\}}\right]+E\left[\left|h_{1}\left(X\right)-h_{1}\left(X'\right)\right|\mathds{1}_{\left\{\left|X-X'\right|\leq\epsilon\right\}}\right]\\
\leq&E\left[\left|h\left(X,Y\right)-h\left(X',Y\right)\right|\mathds{1}_{\left\{\left|X-X'\right|\leq\epsilon\right\}}\right]\\
&+E\left[\left|h\left(X,Y'\right)-h\left(X',Y'\right)\right|\mathds{1}_{\left\{\left|X-X'\right|\leq\epsilon\right\}}\right]\leq 2L\epsilon.
\end{align*}

Obviously, $h_{2,K}$ is $\mathcal{P}$-Lipschitz-continuous with the same constant $2L$ as $h_{2}$. With Lemma \ref{lem6}, choose a random variable $X_{i_{1}}'$ independent of $X_{i_{2}},X_{i_{3}},X_{i_{4}}$ with 
\begin{equation*}
P\left[\left|X_{i_{1}}-X'_{i_{1}}\right|\geq\epsilon\right]\leq18\frac{\left\|X\right\|_{\gamma}^{\frac{\gamma}{2+\gamma}}}{\epsilon^{\frac{\gamma}{2+\gamma}}}\alpha^{\frac{2\gamma}{2+\gamma}}\left(m\right).
\end{equation*}
As $h_{2}$ is a degenerate kernel, we have
\begin{equation*}
E\left[h_{2}\left(X_{i_{1}}',X_{i_{2}}\right)h_{2}\left(X_{i_{3}},X_{i_{4}}\right)\right]=0.
\end{equation*}
Therefore, we get:
\begin{align*}
&\left|E\left[h_{2}\left(X_{i_{1}},X_{i_{2}}\right)h_{2}\left(X_{i_{3}},X_{i_{4}}\right)\right]\right|\\
=&\left|E\left[h_{2}\left(X_{i_{1}},X_{i_{2}}\right)h_{2}\left(X_{i_{3}},X_{i_{4}}\right)\right]-E\left[h_{2}\left(X_{i_{1}}',X_{i_{2}}\right)h_{2}\left(X_{i_{3}},X_{i_{4}}\right)\right]\right|\\
=&\left|E\left[\left(h_{2}\left(X_{i_{1}},X_{i_{2}}\right)-h_{2}\left(X_{i_{1}}',X_{i_{2}}\right)\right)h_{2}\left(X_{i_{3}},X_{i_{4}}\right)\right]\right|\\
\leq&E\left[\left|\left(h_{2,K}\left(X_{i_{1}},X_{i_{2}}\right)-h_{2,K}\left(X_{i_{1}}',X_{i_{2}}\right)\right)h_{2,K}\left(X_{i_{3}},X_{i_{4}}\right)\right|\mathds{1}_{\left\{\left|X_{i_{1}}-X'_{i_{1}}\right|\leq\epsilon\right\}}\right]\\
&+E\left[\left|\left(h_{2,K}\left(X_{i_{1}},X_{i_{2}}\right)-h_{2,K}\left(X_{i_{1}}',X_{i_{2}}\right)\right) h_{2,K}\left(X_{i_{3}},X_{i_{4}}\right)\right|\mathds{1}_{\left\{\left|X_{i_{1}}-X'_{i_{1}}\right|>\epsilon\right\}}\right]\\
&+E\left[\left|h_{2,K}\left(X_{i_{1}},X_{i_{2}}\right)h_{2,K}\left(X_{i_{3}},X_{i_{4}}\right)-h_{2}\left(X_{i_{1}},X_{i_{2}}\right)h_{2}\left(X_{i_{3}},X_{i_{4}}\right)\right|\right]\\
&+E\left[\left|h_{2,K}\left(X'_{i_{1}},X_{i_{2}}\right)h_{2,K}\left(X_{i_{3}},X_{i_{4}}\right)-h_{2}\left(X'_{i_{1}},X_{i_{2}}\right)h_{2}\left(X_{i_{3}},X_{i_{4}}\right)\right|\right]\\
\end{align*}

Because of the $\mathcal{P}$-Lipschitz-continuity and $\left|h_{2,K}\left(X_{3},X_{4}\right)\right|\leq\sqrt{K}$, the first summand is smaller than $2L\epsilon\sqrt{K}$. In consequence of Lemma \ref{lem6}, the second term is bounded by
\begin{equation*}
P\left[\left|X_{i_{1}}-X'_{i_{1}}\right|\geq\epsilon\right]2K\leq36\frac{\left\|X\right\|_{\gamma}^{\frac{\gamma}{2+\gamma}}}{\epsilon^{\frac{\gamma}{2+\gamma}}}\alpha^{\frac{2\gamma}{2+\gamma}}\left(m\right)K.
\end{equation*}
For the third summand, we get:
\begin{multline*}
\quad E\left[\left|h_{2,K}\left(X_{i_{1}},X_{i_{2}}\right)h_{2,K}\left(X_{i_{3}},X_{i_{4}}\right)-h_{2}\left(X_{i_{1}},X_{i_{2}}\right)h_{2}\left(X_{i_{3}},X_{i_{4}}\right)\right|\right]\\
\shoveleft\leq E\left[\left(\left|h_{2}\left(X_{i_{1}},X_{i_{2}}\right)\right|-\sqrt{K}\right)\left|h_{2}\left(X_{i_{3}},X_{i_{4}}\right)\right|\right.\\
\left.\shoveright{\mathds{1}_{\left\{\left|h_{2}\left(X_{i_{1}},X_{i_{2}}\right)\right|>\sqrt{K},\left|h_{2}\left(X_{i_{3}},X_{i_{4}}\right)\right|\leq\sqrt{K}\right\}}\right]}\\
\shoveleft\quad+E\left[\left|h_{2}\left(X_{i_{1}},X_{i_{2}}\right)\right|\left(\left|h_{2}\left(X_{i_{3}},X_{i_{4}}\right)\right|-\sqrt{K}\right)\right.\\
\left.\shoveright{\mathds{1}_{\left\{\left|h_{2}\left(X_{i_{1}},X_{i_{2}}\right)\right|\leq\sqrt{K},\left|h_{2}\left(X_{i_{3}},X_{i_{4}}\right)\right|>\sqrt{K}\right\}}\right]}\\
\shoveleft\quad+E\left[\left|\left(\left|h_{2}\left(X_{i_{1}},X_{i_{2}}\right)\right|-\sqrt{K}\right)\left(\left|h_{2}\left(X_{i_{3}},X_{i_{4}}\right)\right|-\sqrt{K}\right)\right|\right.\\
\left.\shoveright{\mathds{1}_{\left\{\left|h_{2}\left(X_{i_{1}},X_{i_{2}}\right)\right|>\sqrt{K},\left|h_{2}\left(X_{i_{3}},X_{i_{4}}\right)\right|>\sqrt{K}\right\}}\right]}\displaybreak[0]\\
\shoveleft\leq E\left[\left(\left|h_{2}\left(X_{i_{1}},X_{i_{2}}\right)\right|-\sqrt{K}\right)\sqrt{K}\mathds{1}_{\left\{\left|h_{2}\left(X_{i_{1}},X_{i_{2}}\right)\right|>\sqrt{K}\right\}}\right]\\
+E\left[\left(\left|h_{2}\left(X_{i_{3}},X_{i_{4}}\right)\right|-\sqrt{K}\right)\sqrt{K}\mathds{1}_{\left\{\left|h_{2}\left(X_{i_{3}},X_{i_{4}}\right)\right|>\sqrt{K}\right\}}\right]\\
+\frac{1}{2}E\left[\left(\left|h_{2}\left(X_{i_{1}},X_{i_{2}}\right)\right|-\sqrt{K}\right)^{2}\mathds{1}_{\left\{\left|h_{2}\left(X_{i_{1}},X_{i_{2}}\right)\right|>\sqrt{K}\right\}}\right]\\
+\frac{1}{2}E\left[\left(\left|h_{2}\left(X_{i_{3}},X_{i_{4}}\right)\right|-\sqrt{K}\right)^{2}\mathds{1}_{\left\{\left|h_{2}\left(X_{i_{3}},X_{i_{4}}\right)\right|>\sqrt{K}\right\}}\right]\\
\shoveleft\leq\frac{1}{2}E\left[h_{2}^{2}\left(X_{i_{1}},X_{i_{2}}\right)\mathds{1}_{\left\{\left|h_{2}\left(X_{i_{1}},X_{i_{2}}\right)\right|>\sqrt{K}\right\}}\right]\\
\shoveright{+\frac{1}{2}E\left[h_{2}^{2}\left(X_{i_{3}},X_{i_{4}}\right)\mathds{1}_{\left\{\left|h_{2}\left(X_{i_{3}},X_{i_{4}}\right)\right|>\sqrt{K}\right\}}\right]}\displaybreak[0]\\
\leq\frac{1}{2}\frac{E\left|h_{2}\left(X_{i_{1}},X_{i_{2}}\right)\right|^{2+\delta}}{K^{\frac{\delta}{2}}}+\frac{1}{2}\frac{E\left|h_{2}\left(X_{i_{3}},X_{i_{4}}\right)\right|^{2+\delta}}{K^{\frac{\delta}{2}}}\leq\frac{M}{K^{\frac{\delta}{2}}}
\end{multline*}
After treating the fourth summand in the same way, we totally get:
\begin{multline*}
\left|E\left[h_{2}\left(X_{i_{1}},X_{i_{2}}\right)h_{2}\left(X_{i_{3}},X_{i_{4}}\right)\right]\right|\\
\leq2L\epsilon\sqrt{K}+36\frac{\left\|X\right\|_{\gamma}^{\frac{\gamma}{2+\gamma}}}{\epsilon^{\frac{\gamma}{2+\gamma}}}\alpha^{\frac{2\gamma}{2+\gamma}}\left(m\right)K+2\frac{M}{K^{\frac{\delta}{2}}}=:f\left(\epsilon,K\right)
\end{multline*}
Setting $\epsilon^{0}=\left\|X_{1}\right\|_{\gamma}^{\frac{\gamma}{3\gamma+1}}L^{-\frac{2\gamma+1}{3\gamma+1}}\alpha\left(m\right)^{\frac{2\gamma}{3\gamma+1}}K^{\frac{\gamma+\frac{1}{2}}{3\gamma+1}}$, we obtain:
\begin{equation*} f\left(\epsilon^{0},K\right)=38\left\|X_{1}\right\|_{\gamma}^{\frac{\gamma}{3\gamma+1}}L^{\frac{\gamma}{3\gamma+1}}K^{\frac{\frac{5}{2}\gamma+1}{3\gamma+1}}\alpha^{\frac{2\gamma}{3\gamma+1}}\left(m\right)+2\frac{M}{K^{\frac{\delta}{2}}}
\end{equation*}
With $K^{0}=\left\|X_{1}\right\|_{\gamma}^{-\frac{2\gamma}{3\gamma\delta+\delta+5\gamma+2}}L^{-\frac{2\gamma}{3\gamma\delta+\delta+5\gamma+2}}\alpha\left(m\right)^{-\frac{4\gamma}{3\gamma\delta+\delta+5\gamma+2}}M^{\frac{6\gamma+2}{3\gamma\delta+\delta+5\gamma+2}}$, we get the bound:
\begin{multline*}
\left|E\left[h_{2}\left(X_{i_{1}},X_{i_{2}}\right)h_{2}\left(X_{i_{3}},X_{i_{4}}\right)\right]\right|\\
\leq f\left(\epsilon^{0},K^{0}\right)=40\left\|X_{1}\right\|_{\gamma}^{\frac{\gamma\delta}{3\gamma\delta+\delta+5\gamma+2}}L^{\frac{\gamma\delta}{3\gamma\delta+\delta+5\gamma+2}}M^{\frac{5\gamma+2}{3\gamma\delta+\delta+5\gamma+2}}\alpha\left(m\right)^{\frac{2\gamma\delta}{3\gamma\delta+\delta+5\gamma+2}}
\end{multline*}
\end{proof}

\begin{proof}[Proof of Lemma \ref{lem9}:]
The proof is exactly the same as of Yoshihara's Lemma \ref{lem2}, using Lemma \ref{lem8} instead of \ref{lem1}. Therefore, we concentrate on the case $i_{1}<i_{2}<i_{3}<i_{4}$ and $i_{2}-i_{1}\geq i_{4}-i_{3}$. If $i_{2}-i_{1}=m$, there are at most $n$ possibilities for $i_{1}$ and $i_{3}$ and $m$ possibilities for $i_{4}$:
\begin{multline*}
 \sum_{\substack{i_{1}<i_{2}<i_{3}<i_{4}\\i_{2}-i_{1}\geq i_{4}-i_{3}}}\left|E\left[h_{2}\left(X_{i_{1}},X_{i_{2}}\right)h_{2}\left(X_{i_{3}},X_{i_{4}}\right)\right]\right|\leq n^{2}\sum_{m=1}^{n}mK\alpha\left(m\right)^{\frac{2\gamma\delta}{3\gamma\delta+\delta+5\gamma+2}}\\
\leq K_{2}n^{2}\sum_{m=1}^{n}m^{1-\rho\frac{2\gamma\delta}{3\gamma\delta+\delta+5\gamma+2}}=O\left(n^{3-\eta}\right)
\end{multline*}
With a similar argument for the other cases, we get
\begin{equation*}
 \sum_{i_{1},i_{2},i_{3},i_{4}=1}^{n}\left|E\left[h_{2}\left(X_{i_{1}},X_{i_{2}}\right)h_{2}\left(X_{i_{3}},X_{i_{4}}\right)\right]\right|=O\left(n^{3-\eta}\right).
\end{equation*}
\end{proof}

\begin{proof}[Proof of Lemma \ref{lem4}:]
The bootstrapped expectation of $h_{2}\left(X_{i_{1}}^{\star},X_{i_{2}}^{\star}\right)h_{2}\left(X_{i_{3}}^{\star},X_{i_{4}}^{\star}\right)$ (conditionally on $\left(X_{n}\right)_{n\in\mathbb{N}}$) depends on the way the indices $i_{1},i_{2},i_{3},i_{4}$ are allocated to the different blocks. First consider indices $i_{1},i_{2},i_{3},i_{4}$ lying in different blocks (therefore, $X_{i_{1}}^{\star},\ldots,X_{i_{4}}^{\star}$ are independent for fixed $\left(X_{n}\right)_{n\in\mathbb{N}}$). Then the bootstrapped expectation of $h_{2}\left(X_{i_{1}}^{\star},X_{i_{2}}^{\star}\right)h_{2}\left(X_{i_{3}}^{\star},X_{i_{4}}^{\star}\right)$ is a von Mises-statistic and we get
\begin{align*}
&\left|E\left[E^{\star}\left[h_{2}\left(X_{i_{1}}^{\star},X_{i_{2}}^{\star}\right)h_{2}\left(X_{i_{3}}^{\star},X_{i_{4}}^{\star}\right)\right]\right]\right|\\
=&\left|E\left[\frac{1}{n^{4}}\sum_{i_{1},i_{2},i_{3},i_{4}=1}^{n}h_{2}\left(X_{i_{1}},X_{i_{2}}\right)h_{2}\left(X_{i_{3}},X_{i_{4}}\right)\right]\right|\\
\leq&\frac{1}{n^{4}}\sum_{i_{1},i_{2},i_{3},i_{4}=1}^{n}\left|E\left[h_{2}\left(X_{i_{1}},X_{i_{2}}\right)h_{2}\left(X_{i_{3}},X_{i_{4}}\right)\right]\right|.
\end{align*}
There are at most $n^{4}$ possibilities for the four indices to be in four different blocks, so
\begin{multline*}
 \sum_{\substack{i_{1},i_{2},i_{3},i_{4}\\ \mbox{\tiny{4 diff. blocks}}}}\left|E\left[E^{\star}\left[h_{2}\left(X_{i_{1}}^{\star},X_{i_{2}}^{\star}\right)h_{2}\left(X_{i_{3}}^{\star},X_{i_{4}}^{\star}\right)\right]\right]\right|\\
\leq\sum_{i_{1},i_{2},i_{3},i_{4}=1}^{n}\left|E\left[h_{2}\left(X_{i_{1}},X_{i_{2}}\right)h_{2}\left(X_{i_{3}},X_{i_{4}}\right)\right]\right|.
\end{multline*}
As an example, let $i_{1}$ and $i_{2}$ now lie in the same block (write $i_{1}\sim i_{2}$) with $i_{2}-i_{1}=k$, while $i_{3}$, $i_{4}$ lie in two further blocks. The bootstrapped expectation is no longer a von Mises-statistic, as $X_{i_{1}}^{\star}$ and $X_{i_{2}}^{\star}$ are dependent. To repair this, add up the expected values for all $i_{2}$ in the same block as $i_{1}$ and take into account that there are at most $n^{3}$ possibilities for $i_{1},i_{3},i_{4}$:
\begin{multline*}
\quad\left|E\left[E^{\star}\left[h_{2}\left(X_{i_{1}}^{\star},X_{i_{2}}^{\star}\right)h_{2}\left(X_{i_{3}}^{\star},X_{i_{4}}^{\star}\right)\right]\right]\right|\\
\leq\frac{1}{n^{3}}\sum_{i_{1},i_{3},i_{4}=1}^{n}\left|E\left[h_{2}\left(X_{i_{1}},X_{i_{1}+k}\right)h_{2}\left(X_{i_{3}},X_{i_{4}}\right)\right]\right|\\ 
\shoveleft\Rightarrow\sum_{\substack{i_{2}\\i_{1}\sim i_{2}}}\left|E\left[E^{\star}\left[h_{2}\left(X_{i_{1}}^{\star},X_{i_{2}}^{\star}\right)h_{2}\left(X_{i_{3}}^{\star},X_{i_{4}}^{\star}\right)\right]\right]\right|\\
\leq \frac{1}{n^{3}}\sum_{i_{1},i_{2},i_{3},i_{4}=1}^{n}\left|E\left[h_{2}\left(X_{i_{1}},X_{i_{2}}\right)h_{2}\left(X_{i_{3}},X_{i_{4}}\right)\right]\right|\\
 \shoveleft\Rightarrow\sum_{\substack{i_{1},i_{2},i_{3},i_{4}\\i_{1}\sim i_{2}}}\left|E\left[E^{\star}\left[h_{2}\left(X_{i_{1}}^{\star},X_{i_{2}}^{\star}\right)h_{2}\left(X_{i_{3}}^{\star},X_{i_{4}}^{\star}\right)\right]\right]\right|\\
\leq\sum_{i_{1},i_{2},i_{3},i_{4}=1}^{n}\left|E\left[h_{2}\left(X_{i_{1}},X_{i_{2}}\right)h_{2}\left(X_{i_{3}},X_{i_{4}}\right)\right]\right|
\end{multline*}
When the indices are allocated to the blocks in another way, analogous arguments can be used. Totally, we get by Lemma \ref{lem2} or \ref{lem9}, keeping in mind that $\frac{bl}{n}\rightarrow 1$:
\begin{multline*}
E\left[E^{\star}\left[\left(\frac{2}{\sqrt{bl}\left(bl-1\right)}\sum_{1\leq i<j\leq bl}h_{2}\left(X_{i}^{\star},X_{j}^{\star}\right)\right)^{2}\right]\right]\\
\leq\frac{4}{bl\left(bl-1\right)^{2}}\sum_{i_{1},i_{2},i_{3},i_{4}=1}^{n}\left|E\left[E^{\star}\left[h_{2}\left(X_{i_{1}}^{\star},X_{i_{2}}^{\star}\right)h_{2}\left(X_{i_{3}}^{\star},X_{i_{4}}^{\star}\right)\right]\right]\right|\\
\leq\frac{K}{bl\left(bl-1\right)^{2}}\sum_{i_{1},i_{2},i_{3},i_{4}=1}^{n}\left|E\left[h_{2}\left(X_{i_{1}},X_{i_{2}}\right)h_{2}\left(X_{i_{3}},X_{i_{4}}\right)\right]\right|=O\left(n^{-\eta}\right)
\end{multline*}
 
\end{proof}

\begin{proof}[Proof of Lemma \ref{lem5}:]
We use similar arguments as above. If $i_{1},\ldots,i_{8}$ are in 8 different blocks, then the bootstrapped expectation is bounded by
\begin{multline*}
\left|E\left[E^{\star}\left[h_{2}\left(X_{i_{1}}^{\star},X_{i_{2}}^{\star}\right)h_{2}\left(X_{i_{3}}^{\star},X_{i_{4}}^{\star}\right)h_{2}\left(X_{i_{5}}^{\star},X_{i_{6}}^{\star}\right)h_{2}\left(X_{i_{7}}^{\star},X_{i_{8}}^{\star}\right)\right]\right]\right|\\
\leq \frac{1}{n^{8}}\sum_{i_{1},\ldots,i_{8}=1}^{n}\left|E\left[h_{2}\left(X_{i_{1}},X_{i_{2}}\right)h_{2}\left(X_{i_{3}},X_{i_{4}}\right)h_{2}\left(X_{i_{5}},X_{i_{6}}\right)h_{2}\left(X_{i_{7}},X_{i_{8}}\right)\right]\right|.
\end{multline*}
Let now lie $i_{1}$ and $i_{2}$ in the same block and the other indices in different blocks. Then add up the expectations for all $i_{2}$ in the same block as $i_{1}$:
\begin{multline*}
\left|\sum_{i_{2}}E\left[E^{\star}\left[h_{2}\left(X_{i_{1}}^{\star},X_{i_{2}}^{\star}\right)h_{2}\left(X_{i_{3}}^{\star},X_{i_{4}}^{\star}\right)h_{2}\left(X_{i_{5}}^{\star},X_{i_{6}}^{\star}\right)h_{2}\left(X_{i_{7}}^{\star},X_{i_{8}}^{\star}\right)\right]\right]\right|\\
\leq \frac{1}{n^{7}}\sum_{i_{1},\ldots,i_{8}=1}^{n}\left|E\left[h_{2}\left(X_{i_{1}},X_{i_{2}}\right)h_{2}\left(X_{i_{3}},X_{i_{4}}\right)h_{2}\left(X_{i_{5}},X_{i_{6}}\right)h_{2}\left(X_{i_{7}},X_{i_{8}}\right)\right]\right|
\end{multline*}
Treating the other cases in the same way to obtain by Lemma \ref{lem3}
\begin{multline*}
E\left[E^{\star}\left[\left(bl\right)^{2}U_{n}^{\star 4}\left(h_{2}\right)\right]\right]\\
\shoveleft{\leq\frac{K}{\left(bl\right)^{2}\left(bl-1\right)^{4}}}\\
\shoveright{\sum_{i_{1},\ldots,i_{8}=1}^{n}\left|E\left[h_{2}\left(X_{i_{1}},X_{i_{2}}\right)h_{2}\left(X_{i_{3}},X_{i_{4}}\right)h_{2}\left(X_{i_{5}},X_{i_{6}}\right)h_{2}\left(X_{i_{7}},X_{i_{8}}\right)\right]\right|}\\
=O\left(n^{-1-\eta'}\right).
\end{multline*}
\end{proof}

\subsection{U-Statistic CLT}

\begin{proof}[Proof of Theorem \ref{prop4}:] Under the absolute regularity condition, this is Theorem 1 of Yoshihara \cite{yosh}. Under the strong mixing condition, we use the Hoeffding-decomposition:
\begin{equation*}
\sqrt{n}\left(U_{n}\left(h\right)-\theta\right)=\frac{2}{\sqrt{n}}\sum_{i=1}^{n}h_{1}\left(X_{i}\right)+\sqrt{n}U_{n}\left(h_{2}\right)
\end{equation*}
The first summand has a normal limit with variance $4\sigma_{\infty}^{2}$ by Theorem 1.7 of Ibragimov \cite{ibra}. The second summand converges in probability to zero  because of Lemma \ref{lem9}. The theorem follows with the Lemma of Slutzky. 
\end{proof}

\subsection{Bootstrapping U-Statistics}

\begin{proof}[Proof of Theorem \ref{prop3}:] Use the Hoeffding-decomposition
\begin{equation*}
 U^{\star}_{n}\left(h\right)=\theta+\frac{2}{bl}\sum_{i=1}^{bl}h_{1}\left(X_{i}^{\star}\right)+U_{n}^{\star}\left(h_{2}\right)
\end{equation*}
By Theorem 2.3 of Shao, Yu \cite{shao}:
\begin{equation*}
\left|\var^{\star}\left[\frac{2}{\sqrt{bl}}\sum_{i=1}^{bl}h_{1}\left(X_{i}^{\star}\right)\right]-\var\left[\frac{2}{\sqrt{n}}\sum_{i=1}^{n}h_{1}\left(X_{i}\right)\right]\right|\xrightarrow{a.s.}0
\end{equation*}
By the Lemmas \ref{lem2} or \ref{lem9} and \ref{lem4}:
\begin{align}
\var\left[\sqrt{n}U_{n}\left(h_{2}\right)\right]&\xrightarrow{n\rightarrow\infty}0\\
\var^{\star}\left[\sqrt{bl}U_{n}^{\star}\left(h_{2}\right)\right]&\xrightarrow{\mathcal{P}}0\label{line3}
\end{align}
This together proves line (\ref{line1}). To prove line (\ref{line2}), note that for every subsequence of $\left(\var^{\star}\left[\sqrt{bl}U_{n}^{\star}\left(h_{2}\right)\right]\right)_{n\in\mathbb{N}}$, there exists another almost sure convergent subsequence $\left(n_{k}\right)_{k\in\mathbb{N}}
$, and by the Lemma of Slutzky 
\begin{multline*}
\sup_{x\in\mathbb{R}}\left|P^{\star}\left[\sqrt{b_{n_{k}}l_{n_{k}}}\left(U_{n_{k}}^{\star}\left(h\right)-E^{\star}\left[U_{n_{k}}^{\star}\left(h\right)\right]\right)\leq x\right]\right.\\
\left.-P^{\star}\left[\frac{2}{\sqrt{b_{n_{k}}l_{n_{k}}}}\sum_{i=1}^{b_{n_{k}}l_{n_{k}}}\left(h_{1}\left(X_{i}^{\star}\right)-E^{\star}\left[h_{1}\left(X^{\star}_{1}\right)\right]\right)\leq x\right]\right|\xrightarrow{a.s.}0.
\end{multline*}
From Lemma \ref{lem2} or \ref{lem9} and the Lemma of Slutzky follows:
\begin{equation*}
\sup_{x\in\mathbb{R}}\left|P\left[\sqrt{n}\left(U_{n}\left(h\right)-\theta\right)\leq x\right]-P\left[\frac{2}{n}\sum_{i=1}^{n}h_{1}\left(X_{i}\right)\leq x\right]\right|\xrightarrow{n\rightarrow\infty}0
\end{equation*}
With Theorem 2.4 of Shao, Yu \cite{shao} and the triangle inequality, (\ref{line2}) holds for the subsequence $\left(n_{k}\right)_{k\in\mathbb{N}}$ almost surely. Since the subsequence is arbitrary, (\ref{line2}) holds in probability.
\end{proof}

\begin{proof}[Proof of Theorem \ref{prop5}:]
We get from Lemma \ref{lem5} and the Chebyshev inequality
\begin{equation*}
P\left[\var^{\star}\left[\sqrt{bl}U_{n}^{\star}\left(h_{2}\right)\right]>\epsilon\right]\leq\frac{1}{\epsilon^{2}}E\left[n^{2}U_{n}^{4}\left(h_{2}\right)\right]=O\left(n^{-1-\eta'}\right).
\end{equation*}
As these probabilities are summable, the convergence in line (\ref{line3}) holds almost surely under this conditions.
\end{proof}

\begin{proof}[Proof of Corollary \ref{cor1}:]
By Theorem 3.3 of Lahiri \cite{lahi}, the rate of convergence follows for the variance of $\frac{2}{\sqrt{bl}}\sum_{i=1}^{bl}h_{1}\left(X^{\star}_{i}\right)$. The faster convergence to zero of $\left(bl\right)^{2}U_{n}^{\star 4}\left(h_{2}\right)$ (Lemma \ref{lem5}) completes the proof.
\end{proof}

\section{Acknowledgment}
We want to thank two anonymous referees for their very careful reading of the paper and for their helpful suggestions, which lead to a significant improvement of the paper.

\small{

}
\end{document}